\documentclass[11pt, dvipsnames]{amsart}
\usepackage{amsthm}
\usepackage{amssymb}
\usepackage{amsmath}

\usepackage[english]{babel}
\usepackage{tikz}
\usepackage[shortlabels]{enumitem}
\usepackage{url}
\usepackage{hyperref}

\theoremstyle{plain}
\newtheorem{thm}{Theorem}[section]
\newtheorem{prop}[thm]{Proposition}
\newtheorem{lemma}[thm]{Lemma}
\newtheorem{cor}[thm]{Corollary}
\newtheorem{mainthm}{Main Theorem}
\theoremstyle{definition}
\newtheorem{df}[thm]{Definition}
\newtheorem{remark}[thm]{Remark}
\newtheorem{example}[thm]{Example}

\def\Q{\mathbb{Q}}
\def\Z{\mathbb{Z}}
\newcommand{\Wk}{\mathcal{W}}
\DeclareMathOperator{\ICM}{ICM}
\DeclareMathOperator{\Pic}{Pic}
\DeclareMathOperator{\Hom}{Hom}
\DeclareMathOperator{\End}{End}
\DeclareMathOperator{\Ext}{Ext}
\DeclareMathOperator{\Tr}{Tr}

\DeclareMathOperator{\tr}{tr}

\newcommand{\frf}{{\mathfrak f}}
\newcommand{\frm}{{\mathfrak m}}
\newcommand{\p}{{\mathfrak p}}

\newcommand{\frP}{{\mathfrak P}}
\newcommand{\cO}{{\mathcal O}}
\newcommand{\cQ}{{\mathcal Q}}
\renewcommand{\bar}{\overline}
\newcommand{\vphi}{{\varphi}}
\newcommand{\set}[1]{\left\lbrace#1\right\rbrace }
\newcommand{\gensover}[2]{g_{#1}(#2)} 
\newcommand{\gens}[1]{g(#1)} 
\newcommand{\typep}[1]{\mathrm{type}_{\p}({#1})} 
\newcommand{\type}[1]{\mathrm{type}({#1})} 
\newcommand{\typeover}[2]{\mathrm{type}_{#1}({#2})} 

\usepackage{color}

\title{Cohen-Macaulay type of orders, generators and ideal classes}
\date{\today}

\author{Stefano Marseglia}
\address{Mathematical Institute, Utrecht University, P.O. Box 80010, 3508 TA, Utrecht, The Netherlands}
\email{s.marseglia@uu.nl}



\begin{document}

\begin{abstract}
    In this paper we study the (Cohen-Macaulay) type of orders over Dedekind domains in \'etale algebras.
    We provide a bound for the type, and give formulas to compute it.
    We relate the type of the overorders of a given order to the size of minimal generating sets of its fractional ideals, generalizing known results for Gorenstein and Bass orders.
    Finally, we give a classification of the ideal classes with multiplicator ring of type $2$, with applications to the computations of the conjugacy classes of integral matrices and the isomorphism classes of abelian varieties over finite fields.
\end{abstract}

\maketitle
\section{Introduction}\
Cohen-Macaulay rings play a very important role in algebraic geometry.
The associated invariants, such as the (Cohen-Macaulay) type, have been widely studied in the literature.
On the other hand, it seems that a large class of Cohen-Macaulay rings coming from number theory, namely orders in number fields, has been so far widely neglected.
In this paper we work on filling this gap, that is, we study orders in terms of their type.

The typical example of an order is the ring of integers $\cO_K$ of a number field $K$.
The order $\cO_K$ is a central object of study for number theorists since the work of Dedekind and Noether.
For example, we know that all the completions of $\cO_K$ at its maximal ideals have type $1$.
In fact, having (local) type $1$ characterizes a larger class of Cohen-Macaulay rings, the so-called Gorenstein rings.
The study of Gorenstein rings started with the seminal `ubiquity' paper by Bass \cite{Bass63}.
Gorenstein orders in $K$ share several properties with $\cO_K$.
For example, every fractional ideal with Gorenstein multiplicator ring $S$ is invertible, or, equivalently, locally principal.
On the other hand, much less is known about non-Gorenstein orders, that is, when there are maximal ideals for which the local type is strictly greater than $1$.

We will now overview the structure of the paper and highlight the main results.
We will work in a more general context than orders in number fields.
We refer the reader to Section \ref{sec:prelim} for details and missing definitions.
Let $Z$ be a Dedekind domain with field of fractions $Q$.
Let $K$ be an \'etale algebra over~$Q$.
Let $S$ be an order in $K$.
In analogy with the ring of rational integers~$\Z$, the maximal ideals of $S$ will be called primes.
Let $\p$ be such a prime.
In Section \ref{sec:CMtypeorder}, we will study the type of $S$ at $\p$, which is defined as the type of the completion $\hat S_\p$ of $S$ at $\p$, that is,
\[ \typep{S}=\dim_{k}\Ext_{\hat S_\p}^{1}(k,\hat S_\p),\]
where $k$ is the residue field of $\hat S_\p$.
We will also use a global notion of type by setting
\[ \type{S} = \max_\p \set{\typep{S}},\]
where $\p$ runs over the primes of $S$.
This is finite quantity because the local types are almost always $1$.
See Proposition~\ref{prop:type_at_inv_p}.
We will produce formulas to compute these types in terms of the trace dual $S^t$ of $S$.
In fact, the presence of such a non-degenerate trace will be crucial for our investigations in the rest of the paper.

The first main results of the paper about the type are summarized here in Main Theorem \ref{thm:main1}.
See Proposition \ref{prop:can_mod_and_type}, Corollary \ref{cor:typeasmax} and Proposition \ref{prop:type_bound} for the complete statements.
\begin{mainthm}\label{thm:main1}
    Let $S$ be an order in $K$. Then
    \begin{itemize}
        \item We have
        \[ \typep{S} = \dim_{S/\p} \dfrac{S^t}{\p S^t}.\]
        \item If $S\neq \cO_K$ then
        \[ \type{S} +1 = \max\set{ \dim_{S/\p}\dfrac{(\p:\p)}{\p} : \p\text{ is a prime of }S }. \]
        \item Assume that $\dim_Q(K)>1$.
        Then the type of an order in $K$ is bounded by $\dim_Q(K)-1$.
        Moreover, this bound is attained: if $p$ is a prime of $Z$ then the order $T=Z+p \cO_K$ satisfies~$\type{T}=\dim_Q(K)-1$.
    \end{itemize}
\end{mainthm}

In Section \ref{sec:typegens}, we study the type of an order $S$ together with the type of all of its overorders $T$, that is, orders such that $S\subseteq T \subseteq \cO_K$.
We relate these numbers with the quantity $\gens{S}$, which is defined as the maximum, over the fractional $S$-ideals $I$, of the size $\gensover{S}{I}$ of a minimal generating set of $I$ as an $S$-module.
The main result of the section is Theorem \ref{thm:equiv_conds} and it is summarized here as Main Theorem \ref{thm:main2}.
We invite the reader to compare it to the case when $S$ is a non-maximal Bass order, that is, such that all the overorders are Gorenstein. See Proposition \ref{prop:Bass}.

\begin{mainthm}\label{thm:main2}
    Let $S$ be a non-maximal order, $\mathcal{S}$ be the set of overorders~$T$ of $S$
    and $d$ a positive integer.
    Then the following are equivalent:
    \begin{itemize}
        \item $1+\max_{T\in \mathcal{S}}\set{\type{T}} = d$.
        \item The $S$-module $\cO_K/S$ is minimally generated by $d-1$ elements.
        \item $g(S) = d$.
    \end{itemize}
\end{mainthm}

The motivation that originally started the investigation in this paper was to improve our understanding of the isomorphism classes of fractional ideals, and consequently, to improve the algorithms to compute them, which are described in \cite{MarsegliaICM}.
In Section \ref{sec:idealclasses} we recall such results, that where originally stated for $\Z$-orders in \'etale $\Q$-algebras, but that also apply to our more general setting.
The Magma \cite{Magma} implementation of such algorithms is available at \url{https://github.com/stmar89/AlgEt}.
As anticipated above, we completely understand the isomorphism classes of fractional ideals with Gorenstein multiplicator ring.
See Corollary \ref{cor:idlclGorBass}.
Also, as mentioned above, and stated precisely in Proposition \ref{prop:Gorenstein}, the type can be seen as a measure of how far an order is from being Gorenstein.
In Section~\ref{sec:dim2}, we focus our attention to the orders that, from this perspective, are close-to-being Gorenstein, that is, have type $2$.

For an order $S$ and a prime $\p$ with $\typep{S}=2$, we give locally at $\p$ a classification of the isomorphism classes of fractional $S$-ideals with multiplicator ring~$S$.
If we globally have $\type{S}=2$, then we can patch the local information to effectively produce representatives of the isomorphism classes of fractional $S$-ideals with multiplicator ring $S$.
The set of such classes is denoted $\ICM_S(S)$ in the following Main Theorem \ref{thm:main3}.
We refer the reader to Theorem \ref{thm:type2} and Corollary \ref{cor:idlclassestype2} for the complete statements.
\begin{mainthm}\label{thm:main3}
    Let $S$ be an order in $K$. Then
    \begin{itemize}
        \item Let $\p$ be a prime of $S$ and $I$ a fractional $S$-ideal with $(I:I)=S$.
        If $\typep{S}=2$ then either $I_\p \simeq S_\p$ or $I_\p \simeq (S^t)_\p$.
        \item Assume that $\type{S}=2$.
        Let $N$ be the number of primes $\p$ of $S$ such that $\typep{S}=2$.
        Then $\vert\ICM_S(S)\vert = 2^N\cdot \vert \Pic(S) \vert$.
    \end{itemize}
\end{mainthm}
In Example \ref{ex:appliations}, we show how to deduce information about $\Z$-conjugacy classes of integral matrices and abelian varieties over finite fields from a computation of isomorphism classes of fractional ideals.

In a forthcoming paper with Caleb Springer, we give two further applications of Main Theorem \ref{thm:main3} to abelian variaties over finite fields.
The first one is a condition that precludes an abelian variety from being self-dual, and hence principally polarized or the Jacobian of a curve.
The second one is a characterization of the group of rational points of the abelian variety in terms of its endomorphism ring.

Finally, in Section \ref{sec:comparison} we make comparisons with other notions of close-to-being Gorenstein occurring in the recent literature.

\subsection*{Acknowledgements}
The author is grateful to Jonas Bergstr\"om and Hendrik W.~Lenstra Jr.~for useful discussions and comments on a preliminary version of the paper.
The author thanks Carel Faber for useful comments.
The author was supported by NWO grant VI.Veni.202.107. 

\section{Preliminaries}\label{sec:prelim}
\subsection{Conventions and notation}\
In the rest of the paper the word ring will mean a non-zero commutative ring with unity.
We will not consider fields as Dedekind Domains.
In particular, Dedekind domains have Krull dimension $1$ and are infinite.

We denote by $\cQ(R)$ total ring of quotients of a ring $R$ and by $R^\times$ the units of $R$.
Given sub-$R$-modules $I$ and $J$ of $\cQ(R)$ we define $I+J$ and $IJ$ in the natural way, and the {\it colon} as
\[ (I:_{\cQ(R)}J) = \set{ x \in \cQ(R) \: \ xJ \subseteq I }. \]
If the total ring of quotient is clear from the context, we will omit it from the notation and write $(I:J)$ for $(I:_{\cQ(R)}J)$.
\begin{lemma}
    \label{lemma:homcolon}
    Let $I$ and $J$ be sub-$R$-modules of $\cQ(R)$ such that $J\cap \cQ(R)^\times$ is non-empty.
    Then $J\cQ(R)=\cQ(R)$ and the natural $R$-linear morphism
    \begin{align*}
        \vphi: (I:J) & \longrightarrow \Hom_R(J,I)\\
        x & \longmapsto (\vphi_x: j \mapsto xj)
    \end{align*}
    is an isomorphism.
\end{lemma}
\begin{proof}
    Let $z\in J\cap \cQ(R)^\times$.
    Then $\cQ(R)=z\cQ(R) \subseteq J\cQ(R)$ and hence we have the equality $\cQ(R)=J\cQ(R)$.
    Assume that for $x\in (I:J)$ we have $\vphi_x=0$.
    Then $xz=0$ implies that $x=0$ and so $\vphi$ is injective.
    To prove the surjectivity we proceed as follows.
    For an $R$-linear morphism $\psi:J\to I$ we get an induced $\cQ(R)$-linear morphism $\tilde\psi:\cQ(R)\to I\cQ(R)$ by setting $\tilde\psi(jq)=\psi(j)q$ for every $j\in J$ and $q\in\cQ(R)$.
    Set $x=\tilde\psi(1)$.
    For every $j\in J$ we have
    \[ jx=j\tilde\psi(1) = \tilde\psi(j) = \psi(j)\in I \]
    which shows that $x\in (I:J)$ and that $\vphi_x=\psi$, completing the proof.
\end{proof}

We say that $I$ is {\it invertible} if $I(R:I)=R$,
and that $I$ is {\it principal} if there exists an element $x\in \cQ(R)^\times$ such that $I=xR$.
Note that if~$I$ is principal, then it is invertible.
The ring $(I:I)$ is called the {\it multiplicator ring} of $I$.
Observe that $R\subseteq (I:I)$.

A version of the following lemma can be found in \cite[Lemma~2.5]{MarsegliaICM} with more restrictive hypothesis, but the same proof.
\begin{lemma}\label{lemma:invmultring}
    Let $I$ be an invertible sub-$R$-module of $\cQ(R)$.
    Then the multiplicator ring $(I:I)$ of $I$ equals $R$.
\end{lemma}
\begin{proof}
    Using the equality $R=I(R:I)$ we obtain
    \[ R(I:I)=(R:I)I(I:I)=(R:I)I=R. \]
    Therefore $(I:I) \subseteq R$, and hence $R=(I:I)$.
\end{proof}

\begin{df}\label{def:num_gens_mod}
    For an finitely generated $R$-module $M$ set
    \[\gensover{R}{M} = \min\set{ d \in \Z_{>0} : M \text{ can be generated by $d$ elements over $R$} }. \]
\end{df}

\subsection{Orders and Trace form}\
Let $Z$ be a Dedekind domain with field of fractions $Q$.
A $Q$-algebra $K$ is called {\it \'etale} if there exist an integer $n \geq 0$ and an extension $Q\subseteq Q'$ such that $K\otimes_Q Q'\simeq Q'^n$.
Note that an \'etale $Q$-algebra is commutative, unitary and finite dimensional.
In fact, a commutative finite dimensional $Q$-algebra $K$ is \'etale if and only if $K$ is isomorphic to a finite product of separable algebraic extensions of $Q$.

Let $K$ be a non-zero \'etale $Q$-algebra.
By a {\it $Z$-lattice} in $K$ we mean a finitely generated sub-$Z$-module $L$ of $K$ that contains a $Q$-basis of $K$, or equivalently, such that $\dim_Q(L\otimes_Z Q)=\dim_Q(K)$.
A {\it $Z$-order} in $K$ is a subring of $K$ which is also a $Z$-lattice in $K$.
When the Dedekind domain~$Z$ is clear from the context, we will omit it from the terminology and talk simply about lattices and orders.
An order $S$ is reduced, Noetherian,
integral over~$Z$,
and has Krull dimension one.
Hence orders are Cohen-Macaulay rings.
Note that $K$ is the total ring of quotients of $S$.
By an {\it overorder} of $S$, we mean an order $S'$ such that $S \subseteq S' \subset K$.

Let $\cO_K$ be the integral closure of $Z$ in $K$.
Since $Z$ is a Dedekind domain and $K$ is \'etale, one shows that $\cO_K$ is an order in $K$ and contains all orders in $K$.
Therefore we will call $\cO_K$ the {\it maximal order} of~$K$.
%


Pick an element $x$ in $K$.
After fixing a basis of $K$ over $Q$, we can represent the multiplication by $x$ on $K$ by a matrix $m_x$ with coefficients in $Q$, whose {\it trace} $\Tr(x)$ does not depend on the choice of the basis.
Since $K$ is \'etale, the induced bilinear form
\[ K \times K \to Q \qquad (x,y) \mapsto \Tr(xy)\]
is non-degenerate.

Let $L$ be a lattice in $K$.
The trace dual lattice $L^t$ of $L$ is defined as
\[ L^t = \set{ x \in K : \Tr(xL)\subseteq Z }. \]
We have a canonical isomorphism of $Z$-modules
\begin{equation}\label{eq:Homtracedual}
     L^t \simeq \Hom_Z(L,Z),
\end{equation}
induced by $x\mapsto (y \mapsto \Tr(xy))$.
Observe that $(L^t)^t = L$.

Given two lattices $L$ and $L'$ in $K$ we define the {\it colon} $(L:L')$ by
\[ (L:L')=\set{ x \in K \ :\ xL'\subseteq L}. \]
Note that there $(L:L')$ is a lattice and contains a non-zero element of $Z$.
Similarly, the {\it multiplicator ring} of $L$ is the order $(L:L)$.
Observe that the isomorphism in \eqref{eq:Homtracedual} is $S$-linear for every subring $S$ of $(L:L)$.

\subsection{Fractional ideals}\
Let $Z$ be a Dedekind domain with field of fractions~$Q$.
Let $S$ be an order in a non-zero \'etale $Q$-algebra $K$.
A {\it fractional $S$-ideal} $I$ is a finitely generated sub-$S$-module of $K$
which is also a lattice.
Note that a finitely generated sub-$S$-module $I$ of $K$, for example an ideal of~$S$, is a fractional $S$-ideal if and only if $I$ contains a non-zero divisor of~$K$.
Given fractional $S$-ideals $I$ and $J$,
by Lemma \ref{lemma:homcolon}, we have canonical identifications
\begin{equation}
    \Hom_S(J,I) \simeq (I:J) \quad\text{and}\quad \End_S(I) \simeq (I:I).
\end{equation}
\begin{lemma}
\label{lemma:frac_idls}
    Let $I$ and $J$ be fractional $S$-ideals.
    \begin{enumerate}[(i)]
        \item \label{lemma:frac_idls:operations} $IJ$, $I+J$, $I \cap J$, $(I:J)$ and $I^t$ are fractional $S$-ideals.
        \item \label{lemma:frac_idls:eltofZ} $I$ contains a non-zero element of $Z$.
        \item \label{lemma:frac_idls:colontrace} $I J^t = (J:I)^t$.
        \item \label{lemma:frac_idls:pontryagin}
        If $J\subseteq I$ then the $Z$-modules $ I/J $ and $J^t/I^t$ are isomorphic.
    \end{enumerate}
\end{lemma}
\begin{proof}
    Part \ref{lemma:frac_idls:operations} is clear.
    For Part \ref{lemma:frac_idls:eltofZ}, since $S$ and $I$ are lattices, then
    the colon $(I:S)=I$ contains a non-zero element of $Z$.
    Part \ref{lemma:frac_idls:colontrace} is a direct verification.

    We now prove Part \ref{lemma:frac_idls:pontryagin}.
    Put $M=I/J$.
    Observe that $M$ is a torsion module over the Dedekind domain $Z$.
    Hence $M$ has finite length as a $Z$-module.
    Note that $E=Q/Z$ is a minimal injective cogenerator of~$Z$, by for example \cite[Examples 19.11.(1)]{Lam99}.
    It is known that every module of finite length over a Dedekind domain is isomorphic to its Matlis dual.
    See for example \cite[Theorem~1.7]{Ooishi76}.
    Hence we have an isomorphism of~$Z$-modules
    \[ M \simeq \Hom_Z(M,E). \]
    To conclude the proof, it is enough to observe that we have an $S$-linear isomorphism
    \begin{equation}\label{eq:Homtracedualquot}
        J^t/I^t \simeq \Hom_Z(M,E)
    \end{equation}
    induced by
    \[ x+I^t \mapsto (y+J \mapsto \Tr(xy)+Z). \]
    See for example \cite[page 59]{BF65rus}.
    %
\end{proof}

The next lemma, which will be used later in the paper, shows that we can always find generators of a fractional ideal which are non-zero divisors.
\begin{lemma}\label{lemma:gens_nonzerodiv}
    Let $I$ be fractional $S$-ideal in $K$, say $I=y_1 S + \ldots + y_r S$.
    Then there are elements $x_0,x_1,\ldots, x_r$ in $I\cap K^\times$ which generate $I$ as an $S$-module.
\end{lemma}
\begin{proof}
    By Lemma \ref{lemma:frac_idls}.\ref{lemma:frac_idls:eltofZ}, the set $I\cap K^\times$ is non-empty, say containing $x_0$.
    Write $K=K_1\times \ldots \times K_n$, with $K_i$ fields extension of $Q$, and denote by $\pi_i$ the projection $K\twoheadrightarrow K_i$.
    Observe that $\pi_i(x_0)\neq 0$ for every $i$.
    Now, for each generator $y_j$ of $I$, set $x_j=y_j-z_jx_0$ for $z_j\in Z$ such that
    \[ \pi_i(y_j-z_jx_0) \neq 0 \]
    for every $i$.
    Observe that such a $z_j$ exists since we have finitely many constrains while $Z$ is infinite.
    Moreover $x_j S + x_0 S = y_j S +x_0 S$ which shows that the elements $x_0,x_1,\ldots, x_r$ generate $I$ over $S$.
\end{proof}

\subsection{Primes}\
Let $Z$ be a Dedekind domain with field of fractions~$Q$.
Let~$S$ be an order in a non-zero \'etale $Q$-algebra $K$.
A {\it prime}\footnote{In \cite{Reiner03} prime ideals of orders are required to be fractional ideals, that is, they coincide with what we call primes.}
 of $S$ is a fractional $S$-ideal which is also a prime ideal of $S$.
We now state some lemmas that are well known.
\begin{lemma}[{\cite[Theorem 22.3]{Reiner03}}]\label{lemma:primesmax}
    The primes of $S$ are precisely the maximal ideals of $S$.
    If $\p$ is a prime of $S$, then the contraction $p$ of $\p$ in $Z$ is a prime of $Z$, and the field $S/\p$ is a finite extension of $Z/p$.
\end{lemma}

\begin{lemma}[{\cite[Theorem 23.1]{Reiner03}}]\label{lemma:Icontprodprimes}
    Let $I\subsetneq S$ be a fractional $S$-ideal.
    Then $I$ contains a product of primes of $S$.
\end{lemma}

\begin{lemma}\label{lemma:multringp}
    Let $\p$ be a prime of $S$.
    Then $S \subsetneq (S:\p)$.
    If~$\p$ is invertible then $(\p:\p)=S$, while if $\p$ is not invertible then $(\p:\p)=(S:\p)$.
\end{lemma}
\begin{proof}
    Let $x$ be a non-zero divisor in $\p$, which exists by Lemma \ref{lemma:frac_idls}.\ref{lemma:frac_idls:eltofZ}.
    By Lemma~\ref{lemma:Icontprodprimes}, we have
    \[ \p \supseteq xS \supseteq \p_1\cdots\p_r \]
    for primes $\p_i$ of $S$.
    We assume that $r$ is minimal.
    If $r=1$ then $\p=xS$.
    Hence $(S:\p)=(1/x) S$ which strictly contains $S$ since $x$ is not a unit of $S$.
    So assume that $r>1$.
    Then there exists $i$ such that $\p_i=\p$.
    We assume that $\p_1=\p$.
    By minimality of~$r$, there exists $y\in \p_2\cdots \p_r \setminus xS$.
    In particular $(y/x) \not\in S$.
    Also, observe that~$(y/x) \in (S:\p)$ by construction.
    Therefore we have concluded the proof of the first part.

    We now prove the second.
    The inclusion $(\p:\p)\subseteq (S:\p)$ holds by definition.
    If $\p$ is invertible then we conclude by Lemma~\ref{lemma:invmultring}.
    Assume that~$\p$ is not invertible, that is, $\p(S:\p) \subsetneq S$.
    Since $\p$ is maximal by Lemma \ref{lemma:primesmax}, one obtains that
    $\p(S:\p)=\p$.
    Therefore we also have~$(S:\p)\subseteq (\p:\p)$.
\end{proof}

\begin{lemma}
    \label{lemma:finmanyprimes:Iinp}
    Every fractional $S$-ideal $I\subset S$ is contained in finitely many primes of $S$.
\end{lemma}
\begin{proof}
    Observe that the prime ideals of $S$ containing $I$ are necessarily fractional ideals,
    hence they are primes by Lemma \ref{lemma:primesmax}.
    In other words, the ring~$S/I$ has Krull dimension $0$, and since it is Noetherian, it is also Artinian.
    Therefore it has only finitely many maximal ideals,
    which correspond to the primes of~$S$ containing $I$.
\end{proof}

\subsection{Localization}\
Let $Z$ be a Dedekind domain with field of fractions~$Q$.
Let~$S$ be a $Z$-order in a non-zero \'etale $Q$-algebra $K$.
Let $\p$ be a prime ideal of~$S$.
We will denote by $S_\p$ the localization of $S$ at $\p$ and by $Z_p$ the localization of $Z$ at the contraction $p$ of $\p$ in $Z$.
Observe that $Z_p$ is a overring of $Z$ and has field of fractions $Q$.
On the other hand, in general, the total ring of quotients of $S_\p$ is not $K$.

\begin{lemma}\label{lemma:loc_ring_quot}
    The total ring of quotients $\cQ(S_\p)$ is a quotient $Q$-algebra of~$K$.
    The restriction to $S$ of quotient morphism $\pi_\p:K \twoheadrightarrow \cQ(S_\p)$ induces the localization morphism $S\to S_\p$.
\end{lemma}
\begin{proof}
    Since $S$ is a Noetherian reduced ring, the set of zero-divisors of $S$ is the union of the finitely many minimal prime ideals $P_1,\ldots,P_n$ of $S$.
    Hence the total ring of quotients $K$ of $S$ satisfies
    \[ K \simeq \cQ(S/P_1) \times \ldots \times \cQ(S/P_n). \]
    After possibly reindexing, assume that $P_1,\ldots,P_s$ are the minimal prime ideals of $S$ which are contained in $\p$.
    We obtain that
    \[ \cQ(S_\p) \simeq \cQ(S_\p/P_1S_\p)\times \ldots \times \cQ(S_\p/P_sS_\p). \]
    Now observe that for each $i=1,\ldots,s$ we have a natural isomorphism $\cQ(S_\p/P_iS_\p)\simeq \cQ(S/P_i)$.
    Hence we get a natural quotient structure
    \[ \pi_\p:K \twoheadrightarrow \cQ(S_\p). \]
    The fact that $\pi_\p$ restricted to $S$ induces the localization morphism is a straightforward verification.
\end{proof}

In what follows we will denote by $I_\p$ the localization of a fractional $S$-ideal~$I$ at the prime $\p$ of $S$.
Observe that $I_\p$ can be naturally considered a sub-$S_\p$-module of $\cQ(S_\p)$.
In fact, by Lemma \ref{lemma:loc_ring_quot}, we get
\[ I_\p = \pi_\p(I)S_\p. \]
In particular, given fractional $S$-ideals $I$ and $J$ we have that $J_\p$ contains a non-zero divisor of $\cQ(S_\p)$.
Therefore by Lemma \ref{lemma:homcolon}
we have a natural isomorphism of $S_\p$-modules
\[ (I_\p:_{\cQ(S_\p)}J_\p) \simeq \Hom_{S_\p}(J_\p,I_\p). \]

\begin{lemma}\label{lemma:loc_idl_operations}
    Let $\p$ be a prime of $S$.
    For fractional $S$-ideals $I$ and $J$, we have the following equalities:
    \begin{gather*}
        (I+J)_\p=I_\p+J_\p, \quad (IJ)_\p = I_\p J_\p, \quad (I\cap J)_\p=I_\p \cap J_\p,\quad\text{and}\\
        \quad (I:_K J)_\p = (I_\p:_{\cQ(S_\p)}J_\p).
    \end{gather*}
\end{lemma}
\begin{proof}
    The statement for addition, product and intersection follows immediately.
    For the colon, we proceed as follows.
    Let $x_1,\ldots,x_r$ be generators of $J$ over $S$.
    By Lemma \ref{lemma:gens_nonzerodiv} we can assume that $x_i \in K^\times$ for each $i$.
    Then
    \[
    (I:J)_\p
    = \left( \bigcap_{i=1}^r(I:x_i S) \right)_\p
    = \bigcap_{i=1}^r\frac{1}{\pi_\p(x_i)}I_\p
    = \bigcap_{i=1}^r \left(I_\p : \pi_\p(x_i) S_\p\right)
    = (I_\p : J_\p).
    \]
\end{proof}

\begin{lemma}\label{lemma:loc_inv_princ}
    Let $I$ be a fractional $S$-ideal.
    Then the following statements are equivalent:
    \begin{enumerate}[(i)]
        \item \label{lemma:loc_inv_princ:1} $I$ is invertible.
        \item \label{lemma:loc_inv_princ:2} for every prime $\p$ of $S$ the localization $I_\p$ is principal.
        \item \label{lemma:loc_inv_princ:3} for every prime $\p$ of $S$ the $S/\p$-vector space $I/\p I$ has dimension $1$.
    \end{enumerate}
\end{lemma}
\begin{proof}
    Assume \ref{lemma:loc_inv_princ:1}, that is, $I(S:I)=S$.
    Then by Lemma \ref{lemma:loc_idl_operations}, for every prime $\p$ of $S$ we have
    \[ 1 = x_1y_1 + \dots + x_ny_n \]
    for elements $x_i\in I_\p$ and $y_i\in (S_\p:I_\p)$.
    Since $S_\p$ is a local ring, there exists an index $j$ such that $x_jy_j$ is a unit in $S_\p$.
    Hence
    \[ I_\p \subseteq x_j y_j I_\p \subseteq x_j S_\p \subseteq I_\p,  \]
    that is, $I_\p=x_j S_\p$.
    Note that $x_j$ is a non-zero divisor.

    Conversely, assume \ref{lemma:loc_inv_princ:2}.
    Then $I_\p$ is invertible for every prime $\p$ of $S$.
    If $I$ were not invertible, then there would be a prime $\p$ of $S$ such that $I(S:I)\subseteq \p$.
    Then localization at $\p$ together with Lemma \ref{lemma:loc_idl_operations} yields a contradiction.

    The equivalence of \ref{lemma:loc_inv_princ:2} and \ref{lemma:loc_inv_princ:3} is a consequence of Nakayama's Lemma.
\end{proof}

\begin{lemma}\label{lemma:pinvpowerofp}
    Let $\p$ be an invertible prime of $S$.
    Then for every fractional $S$-ideal $I$ we have that $I_\p$ is principal.
    If moreover $I\subseteq S$ then $I_\p = (\p S_\p)^n$ for some integer $n\geq 0$.
\end{lemma}
\begin{proof}
    First assume that $I\subseteq S$.
    If $I=S$ then the statement is true.
    So assume that $I \subset S$.
    Put $\frm=\p S_\p$ and $\frm^{-1}=(S_\p : \frm)$.
    By Lemma \ref{lemma:loc_inv_princ} we have that $\frm$ is principal.
    Also, $I$ contains a product of primes of~$S$ by Lemma~\ref{lemma:Icontprodprimes}.
    Hence the localization $I_\p$ contains a power of $\frm$, say~$I_\p\supseteq \frm^n$ with~$n$ minimal.
    We will now show that $I_\p=\frm^n$ by induction.
    If $n=1$ then~$I_\p=\frm$.
    On the other hand, if $n>1$, since $\frm$ is invertible in $S_\p$, then~$I_\p\frm^{-1} \supseteq \frm^{n-1}$.
    By induction we obtain $I_\p\frm^{-1} = \frm^{n-1}$, and hence~$I_\p = \frm^n$.
    Since $I_\p$ is a power of a principal ideal, then it is principal as well.

    Now consider the general case.
    Let $x$ be a non-zero divisor in $(S:I)$ which exists by Lemma \ref{lemma:frac_idls}.\ref{lemma:frac_idls:eltofZ}.
    The previous argument shows that $xI$ is principal, say $xI=yS$.
    Then $I=(y/x)S$ and we conclude.
\end{proof}

\begin{lemma}\label{lemma:finmanyprimes}
    The following statements hold:
    \begin{enumerate}[(i)]
        \item \label{lemma:finmanyprimes:invDVR}
            Let $\p$ be a prime of $S$ and $\frf=(S:\cO_K)$ be the conductor of $S$. Then the following statements are equivalent:
            \begin{enumerate}[(a)]
                \item \label{invDVR:1} $\p$ is coprime to $\frf$, that is, $\p+\frf=S$.
                \item \label{invDVR:2} $S_\p=\cO_{K,\p}$, where $\cO_{K,\p}$ is the localization of $\cO_K$ at $\p$.
                \item \label{invDVR:3} the $S/\p$-vector space $\cO_K/\p \cO_K$ has dimension one.
                \item \label{invDVR:4} $\p$ is invertible.
            \end{enumerate}
        \item \label{lemma:finmanyprimes:finmanynotinv} There are finitely many primes of $S$ which are not invertible.
    \end{enumerate}
\end{lemma}
\begin{proof}
    We first prove Part \ref{lemma:finmanyprimes:invDVR}.
    Observe that $\p$ is coprime to $\frf$ if and only if $\frf_\p=S_\p$, which is equivalent to $S_\p=\cO_{K,\p}$ by Lemma \ref{lemma:loc_idl_operations}.
    Hence $\ref{invDVR:1}~\Leftrightarrow~\ref{invDVR:2}$.
    Also, it is an immediate consequence of Nakayama's Lemma that $\ref{invDVR:2}~\Leftrightarrow~\ref{invDVR:3}$.
    Now, assume \ref{invDVR:2}.
    Then $\p S_\p$ is invertible.
    Since all the localizations of $\p$ at the other primes of $S$ are trivial, by Lemma \ref{lemma:loc_inv_princ} we get that $\p$ is invertible, that is \ref{invDVR:4} holds.
    Conversely, assume \ref{invDVR:4}.
    By Lemma \ref{lemma:pinvpowerofp} $\frf_\p$ is invertible.
    Hence its multiplicator ring equals $S_\p$ by Lemma \ref{lemma:invmultring}.
    But it is easy to see that the multiplicator ring of $\frf$ contains, and hence equals, $\cO_{K}$.
    Therefore we conclude that $S_\p=\cO_{K,\p}$, that is, \ref{invDVR:2} holds.

    Finally, by Part \ref{lemma:finmanyprimes:invDVR} we have that the non-invertible primes of $S$ are precisely the ones containing the conductor $\frf$ of $S$, and by Lemma \ref{lemma:finmanyprimes:Iinp} there is only finitely many of them, which gives us Part \ref{lemma:finmanyprimes:finmanynotinv}.
\end{proof}

Noticeably, in Lemma \ref{lemma:loc_idl_operations} we have not mentioned the trace duals.
The reason is that the in general the localization $I_\p$ of a fractional $S$-ideal $I$ is not a $Z_p$-lattice in $\cQ(S_\p)$, as the next lemma shows.
\begin{lemma}\label{lemma:SpfgoverZp}
    Let $\p$ be a prime of $S$ and $p$ its contraction in $Z$.
    The ring $S_\p$ is a $Z_p$-lattice if and only if $\p$ is the unique prime of $S$ above $p$.
\end{lemma}
\begin{proof}
    Assume that there exists another prime $\p'$ containing $p$.
    Pick $x\in \p'\setminus \p$, and consider $1/x\in S_\p$.
    Then the minimal polynomial of $1/x$ over $Q$ has denominators in $pZ$ and hence $1/x$ is not integral over $Z_p$.
    Hence $S_\p$ is not finitely generated over $Z_p$.

    Conversely, if $\p$ is the unique prime above $p$, then the natural map
    \[ S \otimes_{Z} Z_p \to S_\p\]
    is an isomorphism.
    In particular, since $S \otimes_{Z} Z_p$ is a finitely generated $Z_p$-module, then $S_\p$ is the same.
\end{proof}

\subsection{Completions}\
Let $Z$ be a Dedekind domain with field of fractions~$Q$.
Let~$S$ be a $Z$-order in a non-zero \'etale $Q$-algebra $K$.
For a prime $\p$ of $S$ we denote by $\hat{S}_\p$ the completion of $S$ at $\p$.
Given a fractional $S$-ideal $I$ we put~$\hat I_\p = I\otimes_S\hat S_\p$.

\begin{lemma}\label{lemma:completion}
    Let $\p$ be a prime of $S$ and $p$ its contraction to $Z$.
    The total ring of quotients $\cQ(\hat{S}_\p)$ is an \'etale $\cQ(\hat Z_p)$-algebra.
    The completion $\hat{S}_\p$ is a $\hat Z_p$-order in $\cQ(\hat{S}_\p)$.
    Given a fractional $S$-ideal $I$, we have that $\hat I_\p$ is a fractional $\hat S_\p$-ideal.
\end{lemma}
\begin{proof}
    By \cite[Cor.~Ch.III~\S.13~p.195]{BourbakiCommAlgCh17eng}, since $S\otimes_Z \hat Z_p$ is a complete semilocal ring, we have an isomorphism
    \begin{equation}\label{eq:completioniso}
        S\otimes_Z \hat{Z_p} \simeq \hat{S}_{\p_1}\oplus \ldots \oplus \hat{S}_{\p_s},
    \end{equation}
    where $\p_1,\ldots,\p_s$ are the primes of $S$ above $p$.
    Applying $-\otimes_Z Q$ to this isomorphism shows that $\cQ(\hat{S}_{\p_i})$ is a $\cQ(\hat Z_p)$-algebra for each $i$, such that the induced trace map gives a non-degenerate bilinear form.
    On the other hand, applying $-\otimes_S I$ shows that $\hat I_{\p_i}$ is a fractional $\hat S_{\p_i}$-ideal for each $i$.
\end{proof}

\begin{lemma}\label{lemma:completion_inv}
    Let $I$ be fractional $S$-ideal. Then the following are equivalent.
    \begin{enumerate}[(i)]
        \item $I$ is invertible.
        \item $\hat I_\p$ is a principal fractional $\hat S_\p$-ideal for every prime $\p$ of $S$.
    \end{enumerate}
\end{lemma}
\begin{proof}
    Observe that we have a canonical isomorphism
    \[ \frac{I}{\p I} \simeq \frac{\hat I_\p}{\p \hat I_\p} \]
    of $S/\p\simeq \hat S_\p/\p \hat S_\p$-vector spaces.
    The result follows from Lemma \ref{lemma:loc_inv_princ}.
\end{proof}

The next lemma, together with Equation \eqref{eq:Homtracedual}, shows that taking trace duals commutes with completion.
In particular, for a fractional $S$-ideal $I$ and a prime $\p$ of $S$, the notation $\hat I_\p^t$ is not ambiguous.
\begin{lemma}\label{lemma:completion_trace}
    Given a fractional $S$-ideal $I$ and a prime $\p$ of $S$ with contraction $p$ in $Z$, we have a natural $\hat S_\p$-linear isomorphism
    \[ \Hom_Z (I,Z)\otimes_S \hat S_\p \simeq \Hom_{\hat Z_p}(\hat I_\p,\hat Z_p). \]
\end{lemma}
\begin{proof}
    Since $\hat Z_p$ is a flat $Z$-algebra we get a natural isomorphism
    \begin{equation}\label{eq:completionflat}
        \Hom_Z (I,Z)\otimes_Z \hat Z_p \simeq \Hom_{\hat Z_p} (I\otimes_Z \hat Z_p,\hat Z_p).
    \end{equation}
    By Equation \eqref{eq:completioniso} we get the following decompositions of $S$-modules
    \[ \Hom_Z (I,Z)\otimes_Z \hat Z_p \simeq \bigoplus_{i=1}^s \left( \Hom_Z (I,Z)\otimes_S \hat S_{\p_i} \right) \]
    and
    \[ \Hom_{\hat Z_p} (I\otimes_Z \hat Z_p,\hat Z_p) \simeq \bigoplus_{i=1}^s \Hom_{\hat Z_p} (\hat I_{\p_i},\hat Z_p), \]
    where $\p_1,\ldots,\p_s$ are the primes of $S$ above $p$.
    Since they are both naturally induced by same isomorphism from Equation \eqref{eq:completioniso}, it is not hard to see that the isomorphism in Equation \eqref{eq:completionflat} decompososes into a direct sum of isomorphisms
    \[ \Hom_Z (I,Z)\otimes_S \hat S_{\p_i} \simeq \Hom_{\hat Z_p}(\hat I_{\p_i},\hat Z_p), \]
    as required.
\end{proof}

\section{Cohen-Macaulay type of an order}
\label{sec:CMtypeorder}
We start this section by recalling some definitions.
We refer the reader to~\cite{BH93} for more details.
Let $R$ be a Cohen-Macaulay local ring with residue field $k$.
A finitely generated $R$-module $C$ is called a {\it canonical module} if $\dim_k \Ext_R^i(k,C)$ is $1$ for $i=\dim R$ and $0$ otherwise.
All canonical modules of $R$, if any, are isomorphic.
The {\it type} of $R$ is defined as $\dim_k\Ext_R^{\dim R}(k,R)$.
In this section we will see that it is easy to compute the type for the completion of an order at a prime.
Furthermore, we will define a global notion of type for an order and we will start studying its properties.

Let $Z$ be a Dedekind domain with field of fractions $Q$.
Let $K$ be a non-zero $Q$-\'etale algebra and let $S$ be an order in $K$.

\newpage
\begin{prop}\label{prop:can_mod_and_type}
    Let $\p$ be a prime of $S$.
    Then
    \begin{enumerate}[(i)]
        \item $\hat S^t_\p$ is a canonical module of $\hat S_\p$.
        \item The type of $\hat S_\p$ equals
        \[ \dim_{S/\p}\frac{S^t}{\p S^t}. \]
    \end{enumerate}
\end{prop}
\begin{proof}
    Let $p$ be the contraction of $\p$ in $Z$.
    Since $\hat Z_p$ is a DVR, it is easy to see from the definition that $\hat Z_p\simeq \Hom_{\hat Z_p}(\hat Z_p,\hat Z_p)$ is a canonical module of~$\hat Z_p$.
    Since $Z$ and $S$ have both Krull dimension $1$, from \cite[Thm.~3.3.7.(b)]{BH93} we get that $C=\Hom_{\hat Z_p}(\hat S_\p,\hat Z_p)$ is a canonical module of $\hat S_\p$.
    To conclude, it is enough to observe that $C$ is canonically isomorphic to $\hat S_\p^t$ by Equation~\eqref{eq:Homtracedual}.
    See also \cite[3.4]{JensenThorup15}.
    For the second part, observe that by Nakayama's lemma a minimal generating set of $\hat S^t_\p$ over $\hat S_\p$ has size equal to the dimension of $\hat S^t_\p/\p \hat S^t_\p\simeq S^t/\p S^t$ as a vector space over $\hat S_\p /\p \hat S_\p\simeq S/\p$.
    The statement then follows from \cite[Prop.~3.3.11.(c).(i)]{BH93}.
\end{proof}

Proposition \ref{prop:can_mod_and_type} tells us that the terminology used in the following definition is compatible with the usual notion of type of a local ring.
\begin{df}
    Let $\p$ be a prime of $S$.
    We say that:
    \begin{enumerate}[(i)]
        \item the {\it (Cohen-Macaulay) type of $S$ at $\p$} is
        \[ \typep{S}:=\dim_{S/\p}\frac{S^t}{\p S^t}. \]
        \item the {\it (Cohen-Macaulay) type of $S$} is
        \[ \type{S}:=\max\set{ \typep{S} : \p \text{ is a prime of }S }. \]
    \end{enumerate}
\end{df}
Note that $\type{S}$ is a finite quantity since $\typep{S}=1$ for almost all primes $\p$ of $S$, as we show in the next Proposition.

\begin{prop}\label{prop:type_at_inv_p}
    Let $\p$ be an invertible prime of $S$.
    Then
     \[ \typep{S}=1. \]
     In particular $\type{S}$ is a finite positive integer.
\end{prop}
\begin{proof}
    By Lemma \ref{lemma:pinvpowerofp}, $(S^t)_\p$ is a principal $S_\p$-module.
    This means that the $S/\p$-vector space $S^t/\p S^t$ has dimension $1$.
    Hence we conclude by Proposition~\ref{prop:can_mod_and_type}.
    Since there are only finitely many non-invertible primes by Lemma~\ref{lemma:finmanyprimes}.\ref{lemma:finmanyprimes:finmanynotinv}, we get that $\type{S}$ is finite.
\end{proof}
In Proposition \ref{prop:type_bound}, once more machinery will be available, we will exhibit an upper bound for the type of the orders in $K$.

The type is a way to encode how far an order is from being Gorenstein.
Recall that a Noetherian ring is {\it Gorenstein} if all its localizations have finite injective dimension.
There are several well known equivalent characterizations of Gorenstein orders.
We now recall some of them, which play a role later in the paper.
\begin{prop}\label{prop:Gorenstein}
    The following statements are equivalent:
    \begin{itemize}
        \item $S$ is Gorenstein.
        \item $\type{S}=1$.
        \item $S^t$ is invertible.
        \item Every fractional $S$-ideal $I$ with $S=(I:I)$ is invertible.
    \end{itemize}
\end{prop}
\begin{proof}
    By \cite[Proposition 3.1.19(c)]{BH93} we have that $S_\p$ is Gorenstein if and only if $\hat S_\p$ is so.
    Hence, by \cite[Theorem 3.3.7(a)]{BH93}, $S$ is Gorenstein if and only if, for every prime~$\p$ of $S$, the localization $\hat S^t_\p$ is a principal fractional $\hat S_\p$-ideal, which by Proposition~\ref{prop:can_mod_and_type} is equivalent to have $\type{S}=1$.
    The equivalence of the second and third statements follows from Lemmas \ref{lemma:completion_inv} and \ref{lemma:loc_inv_princ}.
    Finally, the last two statements are equivalent
    because $S=(I:I)$ if and only if $II^t=S^t$, which is an application of Lemma \ref{lemma:frac_idls}.\ref{lemma:frac_idls:colontrace}.
    See also \cite[Proposition 2.7]{BuchmannLenstra94}.
\end{proof}

We conclude this section with formulas to compute the type of an order which will be used later in the paper.
\begin{prop}\label{prop:cmtype_multring}
    Let $\p$ be a non-invertible prime of $S$.
    Then
    \[ \typep{S} +1 = \dim_{S/\p}\frac{(\p:\p)}{\p}. \]
\end{prop}
\begin{proof}
    Using Lemmas \ref{lemma:frac_idls}.\ref{lemma:frac_idls:colontrace} and \ref{lemma:multringp},
    we see that $(\p S^t)^t = (S:\p) = (\p:\p)$.
    Therefore, the quotient
    \[ \frac{(\p S^t)^t}{S} = \frac{(\p:\p)}{S} \]
    is isomorphic (as a $Z$-module) to $S^t/\p S^t$ by Lemma \ref{lemma:frac_idls}.\ref{lemma:frac_idls:pontryagin}.
    Recall that~$S/\p$ is a finite field extension of $Z/p$, where $p$ is the contraction of $\p$ in $Z$, by Lemma \ref{lemma:primesmax}.
    Hence, comparing the dimensions over $Z/p$, we conclude that
    \[ \dim_{S/\p} \frac{(\p:\p)}{\p} = \dim_{S/\p} \left( \frac{(\p:\p)}{S} \oplus \frac S\p \right)
    = \dim_{S/\p} \frac{S^t}{\p S^t} +1.
     \]
\end{proof}
\begin{cor}\label{cor:typeasmax}
    Let $S$ be a non-maximal order. Then
    \[ \type{S} +1 = \max\set{ \dim_{S/\p}\frac{(\p:\p)}{\p} : \p\text{ is a prime of }S }. \]
\end{cor}
\begin{proof}
    If $\p$ is invertible then $\typep{S} = 1$ by Proposition \ref{prop:type_at_inv_p}.
    Hence in the definition of the global type for a non-maximal order it is enough to consider the non invertible primes.
    Also, by Lemma \ref{lemma:invmultring}, if $\p$ is invertible then $S/\p = (\p:\p)/\p$ .
    Hence also in the right hand side of the equality in the statement, we can restrict ourselves to the non-invertible primes of $S$.
    Therefore the conclusion follows from Proposition \ref{prop:cmtype_multring}.
\end{proof}

\section{Type, overorders and minimal number of generators.}\
\label{sec:typegens}
In this section we exhibit connections between the invertibility and the minimal number of generators of a fractional ideal, and the type of its multiplicator ring.

Let $Z$ be a Dedekind domain with field of fractions $Q$ and let $K$ be a non-zero $Q$-\'etale algebra.
All orders in this section are orders in $K$.
\begin{df}\label{def:num_gens_order}
    For an order $S$, put
    \[ \gens{S} = \max\set{ \gensover{S}{I} : I \text{ is a fractional $S$-ideals}}. \]
\end{df}

\begin{lemma}\label{lemma:gensI}
    Let $S$ be an order and $I$ be a fractional $S$-ideal.
    Then
    \[ \gensover{S}{I}\leq\max\set{ 2,\dim_{S/\p} (I/\p I) : \p \text{ prime of }S}, \]
    with equality if and only if $\gensover{S}{I}\neq 1$, that is, $I$ is not principal.
\end{lemma}
\begin{proof}
    By Nakayama's Lemma, for a prime $\p$ of $S$, we have the equality $\gensover{S_\p}{I_\p}=\dim_{S/\p}(I/\p I)$.
    Put $m=\max\set{ \gensover{S_\p}{I_\p} : \p \text{ a prime of }S}$.

    First we consider the case $m>1$.
    Construct $b_1$ in the following way.
    Put
    \[ \mathcal{A}=\set{ \p : \gensover{S_\p}{I_\p}=m } \]
    and consider the natural surjection
    \[ \pi : I \to \dfrac{I}{I\cdot \prod_{\p \in \mathcal{A}}\p} \simeq \prod_{\p \in \mathcal{A}} \frac{I}{\p I}. \]
    Observe that $\mathcal{A}$ is a finite set since $\gensover{S_\p}{I_\p}=1$ for almost all $\p$ by Lemmas~\ref{lemma:finmanyprimes}\ref{lemma:finmanyprimes:finmanynotinv} and \ref{lemma:pinvpowerofp}.
    Pick $b_1\in I$ such that $\pi(b_1)$ is non-zero in each coordinate.
    Possibly by replacing $b_1$ with $b_1+y$ with $y$ a non-zero divisor in~$I\cdot \prod_{\p \in \mathcal{A}}\p$ chosen with a procedure similar to the proof of Lemma \ref{lemma:gens_nonzerodiv}, we can ensure that $b_1$ is in $I\cap K^\times$.
    Now consider the set
    \[ \mathcal{B}=\set{ \p : I \neq b_1S +\p I }. \]
    Observe that $\mathcal{B}$ is finite since it coincides with the support of the finite length $S$-module $I/b_1 S$.
    For every $\p\in \mathcal{B}$ define vectors $b_{1,\p},\ldots, b_{m,\p}$ of $I/\p I$ in the following way:
    let $b_{1,\p}$ the image of $b_1$ via the surjection $I\to I/\p I$ and then take vectors $b_{2,\p},\ldots,b_{m,\p}$, possibly with repetitions, such that $b_{1,\p},\ldots, b_{m,\p}$ generate $I/\p I$ over $S/\p$.
    Note that $\mathcal{A}\subseteq \mathcal{B}$, so $b_{1,\p}\neq 0$ for all $\p$ such that the dimension of $I/\p I$ is maximal, that is, equal to $m$.
    The Chinese remainder theorem gives an isomorphism of rings
    \[ \frac{S}{\prod_{\p \in \mathcal{B}}\p} \simeq \prod_{\p \in \mathcal{B}} \frac{S}{\p}. \]
    By taking tensor products with $I$ over $S$, we obtain
    an $S$-linear isomorphism
    \[ \frac{I}{I\cdot \prod_{\p \in \mathcal{B}}\p} \simeq \prod_{\p \in \mathcal{B}} \frac{I}{\p I}, \]
    which,
    for $i=2,\ldots,m$, allows us to pick $b_i$ in $I$ such that it reduces to $b_{i,\p}$ modulo $\p I$ for each prime~$\p\in\mathcal{B}$.
    One checks using localization that $b_1,\ldots,b_m$ is a set of generators of~$I$ over $S$: if $\p$ is in $\mathcal{B}$ then this is the case by construction, while if $\p$ is not in~$\mathcal{B}$ then $b_1$ is a local generator at $\p$.

    Now we consider the case when $m=1$.
    Pick an element $b_1$ in~$I \cap K^\times$, which is not empty by Lemma~\ref{lemma:frac_idls}.
    Define the set $\mathcal{B}$ as above.
    In this case, note that $\p \in \mathcal{B}$ if and only if $b_1\in \p I$.
    Pick $b_2 \in I$ so that its reduction modulo $\p I$ is non-zero for every $\p \in \mathcal{B}$.
    Then, again by localization, one checks that $I=b_1 S+b_2 S$.
\end{proof}

It is known that, among all fractional $S$-ideals, the one that requires the biggest number of generators is $\cO_K$.
\begin{lemma}[{\cite[Corollary 2.2]{Greither82}}]
    \label{lemma:greither}
    Let $S$ be an order. Then
    \[ \gens{S} \leq \max\set{2,\gensover{S}{\cO_K}}.\]
    The inequality is strict if and only if $S=\cO_K$ and $\cO_K$ is a product of PID, that is, $K$ has class number one.
\end{lemma}

\begin{cor}\label{cor:gensS}
    Let $S$ be a non-maximal order. Then
    \[ \gens{S} = \gensover{S}{\cO_K} = \max\set{ \dim_{S/\p} (\cO_K/\p \cO_K) : \p \text{ prime of }S}. \]
\end{cor}
\begin{proof}
    Since $S\neq \cO_K$ then $\gensover{S}{\cO_K}\geq 2$.
    The first equality follows from Lemma \ref{lemma:greither} and the second follows from Lemma \ref{lemma:gensI}.
\end{proof}

The following proposition allows us to compare the size of minimal generating sets of the fractional ideals of an order with the type of some specific overorders.

\begin{prop}\label{prop:specialoo_numgens}
    Let $\p$ be a prime of an order $S$.
    Then
    \[
    \dim_{S/\p} \frac{\cO_K}{\p \cO_K} =
    \begin{cases}
        1 & \text{if $\p$ is invertible,}\\
        1+\type{S+\p \cO_K} & \text{otherwise.}
    \end{cases}
    \]
    In particular, if $S$ is not maximal then
    \[ g(S) = 1+\max\set{ \type{S+\p \cO_K} : \p \text{ a prime of }S}. \]
\end{prop}
\begin{proof}
    Put $\frP=\p \cO_K$ and consider the overorder $T=S+\frP$ of $S$.
    Observe that we have inclusions $\p \subseteq S \cap \frP \subseteq S$.
    Since, $1$ is not in $\frP$, we have the equality $\p = S \cap \frP$.
    Hence
    \begin{equation}\label{eq:resfield}
        \frac{T}{\frP} = \frac{S+\frP}{\frP} \simeq \frac{S}{\frP \cap S} = \frac{S}{\p},
    \end{equation}
    which, in particular, shows that $\frP$ is a prime of $T$.
    Note that $(\frP:\frP)=\cO_K$ and so
    \begin{equation}\label{eq:OpOvsSp}
        \frac{\cO_K}{\p\cO_K} = \frac{\cO_K}{\frP} = \frac{(\frP:\frP)}{\frP}.
    \end{equation}

    We already know by Lemma \ref{lemma:finmanyprimes}.\ref{lemma:finmanyprimes:invDVR} that
    $\p$ is invertible if and only if~$\cO_K/\p\cO_K$ has dimension $1$ over $S/\p$.
    So we assume that $\p$ is not invertible.
    If $\frP$ were invertible then by Lemma \ref{lemma:invmultring} we would have $(\frP:\frP)=T$.
    Hence by Equations \eqref{eq:resfield} and \eqref{eq:OpOvsSp} we would have that $\cO_K/\p\cO_K$ has dimension $1$ as a vector space over $S/\p$.
    Hence $\frP$ is not invertible.
    So we can apply Proposition \ref{prop:cmtype_multring} which gives us the first equality in
    \[ \typeover{\frP}{T} + 1 = \dim_{T/\frP} \frac{(\frP:\frP)}{\frP} = \dim_{S/\p} \frac{\cO_K}{\p\cO_K}, \]
    while the last equality is given by Equations \eqref{eq:resfield} and \eqref{eq:OpOvsSp}.
    To conclude the proof we need to show that $\typeover{\frP}{T} = \type{T}$.
    Note that $\frP$ is a fractional~$\cO_K$-ideal inside $T$.
    So $\frP$ is contained in the conductor $(T:\cO_K)$, which is a proper ideal since $\frP$ is not invertible and $(T:\cO_K) \subseteq (\frP:\frP)$.
    Hence $\frP=(T:\cO_K)$ by maximality of $\frP$.
    By Lemma \ref{lemma:finmanyprimes}.\ref{lemma:finmanyprimes:invDVR} we have that $\frP$ is the only non invertible ideal of~$T$.
    Proposition \ref{prop:type_at_inv_p} shows then that the only contribution to the global type of $T$ comes from $\frP$, that is, $\typeover{\frP}{T} = \type{T}$.

    The last statement is an immediate consequence of Corollary \ref{cor:gensS}.
\end{proof}

In Section \ref{sec:CMtypeorder}, we discussed Gorenstein orders.
A generalization, which takes into account all the overorders, is the following.
An order which is either maximal or satisfies the equivalent conditions of Proposition \ref{prop:Bass} is called a {\it Bass} order.

\begin{prop}\label{prop:Bass}
    Let $S$ be a non-maximal order. Then the following are equivalent:
    \begin{enumerate}[(i)]
        \item For every overorder $T$ of $S$ we have $\type{T}=1$.
        \item Every overorder $T$ of $S$ is Gorenstein.
        \item Every fractional $S$-ideal $I$ is invertible as a fractional $(I:I)$-ideal.
        \item The $S$-module $\cO_K/S$ is cyclic.
        \item $\gens{S} = 2$.
    \end{enumerate}
\end{prop}
\begin{proof}
    The equivalence of the first three statement is a direct consequence of Proposition \ref{prop:Gorenstein}.
    The equivalence of the last four statements can be found in \cite[Theorem 2.1]{LevyWiegand85}.
\end{proof}

Proposition \ref{prop:Gorenstein} tells us that the type measures the deviation of an order from being Gorenstein.
Following this philosophy, a measure of how an order fails to being Bass can be given in terms of the types of all the overorders.
More precisely, we have the following generalization of Proposition \ref{prop:Bass}.

\begin{thm}\label{thm:equiv_conds}
Let $S$ be a non-maximal order, $\mathcal{S}$ be the set of overorders $T$ of $S$, and $d$ be a positive integer.
Then the following are equivalent:
\begin{enumerate}[(i)]
    \item \label{thm:equiv_conds:1}
        $1+\max_{T\in \mathcal{S}}\set{\type{T}} = d$.
    \item \label{thm:equiv_conds:2}
        $\max_{T\in \mathcal{S}} \set{ \dim_{T/\frP} \left((\frP:\frP)/\frP\right) : \frP\text{ a prime of }T} = d.$
    \item \label{thm:equiv_conds:OKS}
        $\gensover{S}{\cO_K/S}=d-1$.
    \item \label{thm:equiv_conds:3}
        $\max_{T\in \mathcal{S}} \set{ \dim_{T/\frP} \left(\cO_K/\frP\cO_K\right) : \frP\text{ a prime of }T } = d$.
    \item \label{thm:equiv_conds:4}
        $\max_{T\in \mathcal{S}}\set{\gens{T}} = d$.
    \item \label{thm:equiv_conds:5}
        $g(S) = d$.
\end{enumerate}
\end{thm}
\begin{proof}
    Let $T$ be a non-maximal overorder in $\mathcal{S}$. Then
    \[ \type{T}+1 = \max\set{ \dim_{T/\frP} \frac{(\frP:\frP)}{\frP}: \frP\text{ a prime of }T}, \]
    by Corollary \ref{cor:typeasmax}.
    This implies the equivalence of \ref{thm:equiv_conds:1} and \ref{thm:equiv_conds:2}.
    Corollary~\ref{cor:gensS} gives the equivalence of \ref{thm:equiv_conds:3} and \ref{thm:equiv_conds:4}.
    Note that for every $T\in \mathcal{S}$, every fractional $T$-ideal $I$ is also a fractional $S$-ideal.
    So we have $\gensover{T}{I}\leq \gensover{S}{I}$.
    Hence $\gens{S}=\max_{T\in \mathcal{S}}\set{\gens{T}}$.
    In particular, \ref{thm:equiv_conds:4} and \ref{thm:equiv_conds:5} are equivalent.

    We now show that \ref{thm:equiv_conds:OKS} is equivalent to \ref{thm:equiv_conds:5}.
    Denote by $M$ the $S$-module~$\cO_K/S$.
    By Lemma \ref{lemma:finmanyprimes} we have that $M_\p\neq 0$ only for finitely many primes $\p$ of $S$.
    By Nakayama's Lemma and an argument similar to the proof of Lemma \ref{lemma:gensI}, we have
    \[ \gensover{S}{M}=\max\set{ \dim_{S/\p}\left( \dfrac{M}{\p M} \right) :\ \p \text{ a prime of }S }. \]
    Observe that for every prime $\p$ of $S$ we have
    \begin{equation}\label{eq:gensM}
        \dfrac{M}{\p M} \simeq \dfrac{\cO_K}{S+\p\cO_K}.
    \end{equation}
    Moreover, we have a short exact sequence of $S/\p$-vector spaces
    \begin{equation}\label{eq:sesOKS}
        0 \to \dfrac{ S+\p \cO_K }{ \p \cO_K } \to \dfrac{\cO_K}{\p\cO_K} \to \dfrac{\cO_K}{S+\p\cO_K} \to 0.
    \end{equation}
    As shown in Equation \eqref{eq:resfield}, we have that $(S+\p\cO_K)/\p\cO_K$ has dimension $1$ over $S/\p$.
    Hence by Equations \eqref{eq:gensM} and \eqref{eq:sesOKS}, we get
    \[ \dim_{S/\p}\left( \dfrac{\cO_K}{\p\cO_K} \right) = \dim_{S/\p}\left( \dfrac{M}{\p M} \right) + 1.  \]
    By Corollary \ref{cor:gensS}, taking the maximum over the primes of $S$ we get
    \[ \gens{S} = \gensover{S}{M} +1, \]
    as required.

    To conclude we show that \ref{thm:equiv_conds:1} and \ref{thm:equiv_conds:5} are equivalent, that is, that we have an equality
    \[ \gens{S} =  1+\max_{T\in \mathcal{S}}\set{\type{T}}. \]
    By Proposition \ref{prop:specialoo_numgens}, we obtain
    \begin{equation}\label{eq:gensmultring}
    \begin{split}
        \gens{S} & = 1+\max\set{ \type{S+\p\cO_K} :\ \p \text{ a prime of }S}\\
                 & \leq 1+\max_{T \in \mathcal{S} }\set{ \type{T} }.
    \end{split}
    \end{equation}
    For the other inequality, pick any non-maximal $T\in \mathcal{S}$. 
    For any prime $\frP$ of~$T$, we have
    \begin{equation}\label{eq:gensPP}
        \dim_{T/\frP} \frac{(\frP:\frP)}{\frP} \leq \gensover{T}{(\frP:\frP)}\leq \gens{T} \leq \gens{S}, 
    \end{equation}
    where the first inequality follows from Lemma \ref{lemma:gensI}, the second is clear and the third follows from the discussion above.
    Combining Equation \eqref{eq:gensPP} and Corollary \ref{cor:typeasmax}, gives us
    \[ \type{T}+1 \leq \gens{S}, \]
    as required to complete the proof.
\end{proof}
\begin{remark}
    In \cite{MarsegliaSuperMult}, we showed that an order $S$ in a number field satisfies $g(S)\leq 3$ if and only if every overorder $T$ of $S$ is \emph{super-multiplicative}, that is, for every pair of fractional $T$-ideals $I$ and $J$ contained in $T$ we have
    \[ N(IJ)\geq N(I)N(J), \]
    where the ideal norm $N(I)$ (resp.~$N(J)$ and $N(IJ)$) is defined as the size of the finite abelian group $T/I$ (resp.~$T/J$ and $T/IJ$).
    In view of Theorem~\ref{thm:equiv_conds}, we can state that every overorder $T$ of $S$ is supermultiplicative if and only if for every overorder $T$ of $S$ we have $\type{T}\leq 2$.
\end{remark}

We conclude Section \ref{sec:typegens} by showing that the type of the orders in a non-zero \'etale $Q$-algebra $K$ is bounded.
Note that the only order in $Q$ is $Z$, which has type $1$.
To avoid this trivial case, we will assume in the next proposition that $\dim_Q(K)>1$.
\begin{prop}\label{prop:type_bound}
    Assume that $\dim_Q(K)>1$.
    Then the type of an order in $K$ is bounded by $\dim_Q(K)-1$.
    Moreover, this bound is attained: for any prime $p$ of $Z$ the order $T=Z+p \cO_K$ satisfies $\type{T}=\dim_Q(K)-1$.
\end{prop}
\begin{proof}
    Pick $S$ any order in $K$.
    If $S$ is Gorenstein then $\type{S}=1$ by Proposition \ref{prop:Gorenstein}.
    So we assume that $S$ is not Gorenstein.
    By definition, there exists a prime $\p$ of $S$ such that $\type{S}=\typep{S}$.
    By Proposition~\ref{prop:type_at_inv_p}, the prime $\p$ is not invertible.
    Combining Proposition~\ref{prop:cmtype_multring} and Lemma~\ref{lemma:gensI} we obtain $\typep{S} = \gensover{S_\p}{(\p:\p)_\p}-1$.
    Also we have the equality
    \[\gensover{S_\p}{(\p:\p)_\p} = \gensover{\hat S_\p}{(\p:\p)\otimes_S \hat S_\p}.\]
    Let $p$ be the contraction in $Z$ of $\p$.
    Equation~\eqref{eq:completioniso} implies that the $\hat Z_p$-order $\hat S_\p$ is a free $\hat Z_p$-module of rank $r \leq \dim_Q(K)$.
    Hence the fractional $\hat S_\p$-ideal $(\p:\p)\otimes_S \hat S_\p$ is also a free $\hat Z_p$-module of rank $r$.
    It follows that $\gensover{\hat S_\p}{(\p:\p)\otimes_S \hat S_\p} \leq \dim_Q(K)$.
    We conclude that~$\type{S}\leq \dim_Q(K)-1$, as required.

    Now we prove the second part.
    Let~$p$ be any prime of $Z$ and define $T=Z+p\cO_K$ as in the statement.
    Put $F_p=Z/p$.
    We have a ring isomorphism
    \[ \frac{T}{p\cO_K} \simeq \frac{Z}{Z\cap p\cO_K} = F_p, \]
    where the last equality holds because $Z\cap p\cO_K = p$ by maximality of $p$.
    In particular, we get that $p\cO_K$ is a prime of $T$ with residue field isomorphic to~$F_p$.
    Since $\cO_K$ is a free $Z$-module of rank $\dim_Q(K)$, after taking the tensor product over $Z$ with $F_p$, we get $\dim_Q(K) = \dim_{F_p}(\cO_K/p\cO_K)$ which is $>1$ by assumption.
    We conclude that $T\neq \cO_K$.
    Thus the inclusion $p\cO_K \subseteq (T:\cO_K)$ is an equality.
    Hence Lemma~\ref{lemma:finmanyprimes}.\ref{lemma:finmanyprimes:invDVR} tells us that $p\cO_K$ is the unique non-invertible prime of $T$.
    Then Proposition~\ref{prop:type_at_inv_p} gives us that $\type{T} = \typeover{p\cO_K}{T}$.
    Note that the multiplicator ring of~$p\cO_K$ is~$\cO_K$.
    Using Proposition~\ref{prop:cmtype_multring}, we get
    \[ \type{T} = \dim_{F_p}\frac{\cO_K}{p \cO_K}-1 = \dim_Q(K) -1. \]
\end{proof}

\section{ Ideal classes }\
\label{sec:idealclasses}
Let $Z$ be a Dedekind domain with field of fractions $Q$ and let $K$ be a non-zero $Q$-\'etale algebra. Pick an order $S$ in $K$.
Recall that two fractional $S$-ideals $I$ and $J$ are {\it isomorphic} as $S$-modules if and only if there exists $\alpha\in K^\times$ such that $\alpha I=J$.
The set of isomorphism classes of invertible fractional $S$-ideal, the \emph{Picard group} of $S$, which we denote $\Pic(S)$, has a natural group structure induced by fractional ideal multiplication.
Similarly, the set of isomorphism classes of all fractional $S$-ideals forms a commutative monoid that we denote by $\ICM(S)$ and call the \emph{ideal class monoid} of $S$.
The ideal class monoid has a natural partitioning
\[ \ICM(S) = \bigsqcup_T \ICM_T(S), \]
where the disjoint union is taken over the overorders $T$ of $S$, and $\ICM_T(S)$ consists of the classes with multiplicator ring $T$.

We say that two fractional $S$-ideals $I$ and $J$ are \emph{weakly equivalent} if $I_\p \simeq J_\p$ for every prime~$\p$ of~$S$.
This notion was introduced in \cite{DadeTausskyZas} for orders in number fields.
Denote by the $\Wk(S)$ the monoid of weak equivalence classes of fractional~$S$-ideals.
As we will see in the next proposition, the multiplicator ring of a fractional~$S$-ideal is an invariant of its weak equivalence class.
Hence we have a partitioning
\[ \Wk(S)=\bigsqcup_T\Wk_T(S), \]
where the disjoint union is taken over the over-orders $T$ of $S$, and $\Wk_T(S)$ is the subset of $\Wk(S)$ consisting of the classes with multiplicator ring equal to~$T$.

In the next proposition, we summarize some useful facts about isomorphism and weak equivalence classes, and how invertible ideals connect the two notions.
More details can be found in \cite{MarsegliaICM}.
\begin{prop}\label{prop:computeICM}
    Let $S$ be an order.
    \begin{enumerate}[(i)]
        \item \label{prop:computeICM:wkeq}
            Let $I$ and $J$ be two fractional $S$-ideals. Then the following are equivalent
            \begin{itemize}
                \item $I$ and $J$ are weakly equivalent.
                \item $1\in (I:J)(J:I)$.
                \item $I$ and $J$ have the same multiplicator ring, say $T$, and there exists an invertible fractional $T$-ideal $L$ such that $I=LJ$.
            \end{itemize}
        \item \label{prop:computeICM:wkeqclasses}
            Let $T$ be an overorder of $S$.
            Let $T_0$ be the smallest overorder of $T$ such that the extension to $T_0$ of every fractional $T$-ideal with multiplicator ring $T$ is invertible.
            Put $\frf_0=(T:T_0)$.
            Then every class in $\Wk_T(S)$ admits a representative $I$ such that
            \[ \frf_0 \subseteq I \subseteq T_0. \]
        \item \label{prop:computeICM:isoclasses}
            Fix an overorder $T$ of $S$.
            Then there is a free action of $\Pic(T)$ on~$\ICM_T(S)$ with quotient $ \mathcal{W}_T(S)$.
    \end{enumerate}
\end{prop}
\begin{proof}
    The statements for orders over $\Z$ in a non-zero \'etale $\Q$-algebras can be found in \cite[Prop~4.1,  Prop.~5.1, Thm.~4.6]{MarsegliaICM}.
    The same proofs are valid also for the more general case under consideration in this paper.
    Possibly, the only fact used in proving \ref{prop:computeICM:wkeq} that is not obvious in our more general setting is that
    the extension map induces a surjective group homomorphism from $\Pic(S)$ onto $\Pic(T)$.
    The proof of this for the case $T=\cO_K$ can be found in \cite[Sec.~2]{Wieg84}, and it works in the same way in the general case, as stated in \cite[Eq.~(4.3.1)]{LevyWiegand85}.
\end{proof}
\begin{remark}
    Part \ref{prop:computeICM:wkeq} gives a method to verify if two fractional ideals are weakly equivalent as long as we have a method to compute products and colons of fractional ideals, and check inclusion of elements.

    From Part \ref{prop:computeICM:wkeq} and Part \ref{prop:computeICM:wkeqclasses}, we obtain a method to compute $\Wk_T(S)$ if
    the quotient $T_0/\frf_0$ has finitely many sub-$S$-modules and
    we have a procedure to loop over all of them.
    This is the case for example if the quotient of every two $Z$-lattices is a finite set.
    Indeed, in that case we can use algorithms to enumerate subgroups in order to obtain all the submodules. See for example \cite[Sec.~5.2]{FHS19}.

    Then we can use Part \ref{prop:computeICM:isoclasses} to compute $\ICM_T(S)$ if we have an algorithm to compute $\Pic(T)$.
    Such an algorithm for orders over $\Z$ and $\mathbb{F}_q[t]$ is given in~\cite{klupau05}.

    Finally, these steps allow us to compute the whole $\ICM(S)$ if we can compute the overorders $T$ of $S$.
    We refer to \cite{HofmannSircana20} for an effective method.
\end{remark}

Thanks to Proposition \ref{prop:computeICM} we can reformulate the properties of being Gorenstein and Bass for an order in terms of ideal classes.
\begin{cor}\label{cor:idlclGorBass}
    Let $T$ be an overorder of an order $S$.
    Then $T$ is Gorenstein if and only if $\vert\mathcal{W}_T(S)\vert = 1$ if and only if $\Pic(T) = \ICM_T(S)$.
    The order $S$ is Bass if and only if
    \[ \ICM(S) = \bigsqcup_T \Pic(T), \]
    where the disjoint union is taken over the overorders $T$ of $S$.
\end{cor}
In the next Section we prove an analogous result for orders of type~$2$.
More precisely, in Corollary \ref{cor:idlclassestype2} we will show that also in this case we completely understand $\ICM_T(S)$ for an overorder of $S$ of $\type{T}=2$.

\section{Orders with CM type 2}
\label{sec:dim2}
Let $Z$ be a Dedekind domain with field of fractions $Q$ and let $K$ be a non-zero $Q$-\'etale algebra.
Let $S$ be an order in $K$.
We start by recalling a lemma from linear algebra.

\begin{lemma}[{\cite[Lemma 4.9]{MarsegliaSuperMult}}]
    \label{lemma:linalg1}
    Let $U,V$ and $W$ be vector spaces over a field $k$, with $W$ of dimension $\geq 2$. Let $\vphi:U\otimes V \twoheadrightarrow W$ be a surjective linear map. Then there exists an element $u\in U$ such that $\dim_k \vphi(u\otimes V)\geq 2$, or there exists an element $v\in V$ such that $\dim_k \vphi(U\otimes v)\geq 2$.
\end{lemma}

In the next Theorem, given an order $S$ with local type $2$, we give a complete classification of the local isomorphism classes of fractional $S$-ideals with multiplicator ring $S$.
\begin{thm}\label{thm:type2}
    Let $S$ be an order, $\p$ a prime of $S$ and $I$ a fractional $S$-ideal with $(I:I)=S$.
    Assume that
    $\typep{S}=2$.
    Then either~$I_\p \simeq S_\p$ or~$I_\p \simeq (S^t)_\p$.
\end{thm}
\begin{proof}
    Put $k=S/\p$ and consider the following finite dimensional $k$-vector spaces:
    \[ U = \frac{I}{\p I} ,\quad V = \frac{I^t}{\p I^t} ,\quad W = \frac{S^t}{\p S^t}. \]
    By Lemma \ref{lemma:frac_idls}.\ref{lemma:frac_idls:colontrace} we have that $S^t=(I:I)^t=II^t$.
    It follows that the multiplication map $I \otimes_S I^t \to S^t$ is a surjective $S$-linear morphism.
    In particular the induced $k$-linear morphism
    \[ \pi : U \otimes_k V \to W \]
    is a surjective map onto a $2$-dimensional vector space.
    By Lemma \ref{lemma:linalg1} there exists
    \begin{enumerate}[(i)]
        \item \label{thm_surj:1} $u\in I$ such that $\pi( u+\p I \otimes V ) = W$, or
        \item \label{thm_surj:2} $v\in I^t$ such that $\pi( U \otimes v+\p I^t ) = W$.
    \end{enumerate}
    Assume that \ref{thm_surj:1} holds, that is,
    \[ \pi( u+\p I \otimes V )=\frac{uI^t+\p S^t}{\p S^t} = W. \]
    If we write $I^t = \sum_{k=1}^r \beta_k S$ then $\set{ u\beta_k + \p S^t }$ is a finite set of generators over~$k$ for $W$.
    By Nakayama's Lemma
    we obtain that $\set{ u\beta_k }$ gives a set of generators over $\hat S_\p$ for $\hat S^t_\p$.
    By identifying $u$ with its image in $\cQ(\hat S_\p)$ we see that
    $u\hat I^t_\p = \hat S^t_\p$, which is equivalent to $\hat I_\p = u\hat S_\p$.
    By \cite[Ex.~7.5, p.~203]{Eisenbud95} we get that $I_\p\simeq S_\p$ as an $S_\p$-module.
    If \ref{thm_surj:2} holds, then by a similar argument we get that
    $v\hat I_\p = \hat S^t_\p$ and hence $I_\p \simeq (S^t)_\p$.
\end{proof}

\begin{cor}
    Let $S$ be a non-Gorenstein order.
    Then the following are equivalent:
    \begin{itemize}
        \item $\type{S}=2$.
        \item $\max\set{\gensover{S}{I} :\text{ fractional $S$-ideal $I$ with $S=(I:I)$}} = 2$
    \end{itemize}
\end{cor}
\begin{proof}
    Recall that for every prime $\p$ of $S$, we have $\gensover{S_\p}{I_\p}=\dim_{S/\p}(I/\p I)$ for every fractional $S$-ideal $I$ by Nakayama's Lemma.
    Define
    \[ m = \max\set{\gensover{S}{I} :\text{ fractional $S$-ideal $I$ with $(I:I)=S$}}. \]

    Assume that $\type{S}=2$.
    By Lemma \ref{lemma:gensI} this is equivalent to  $\gensover{S}{S^t}=2$.
    Since $(S^t:S^t)=S$ we then have $m\geq 2$.
    Now, let $I$ be any fractional $S$-ideal with $(I:I)=S$.
    By Theorem \ref{thm:type2} we have $\gensover{S_\p}{I_\p}\leq 2$, and hence by Lemma \ref{lemma:gensI}, we also have $\gensover{S}{I} \leq 2$.
    Since $I$ was arbitrary we get also $m\leq 2$.
    Therefore $m=2$, as required.

    For the converse, assume that $m=2$.
    Since $S$ is not Gorenstein, we have that $\gensover{S}{S^t}>1$ and hence $\gensover{S}{S^t}=2$.
    Hence we get that $\type{S}=2$.
\end{proof}

After a technical lemma, we conclude Section \ref{sec:dim2} with an application of the machinery built so far to the computations of isomorphism classes, followed by Example \ref{ex:appliations}.

\begin{lemma}\label{lemma:patchlocaltoglobal}
    Let $S$ be an order and let $\p_1,\ldots,\p_r$ be distinct primes of $S$.
    Let $I_1,\ldots,I_r$ be fractional $S$-ideals contained in $S$, possibly with repetitions.
    Let $m_1,\ldots,m_r$ be non-negative integers such that
    \[ \p_i^{m_i}\subseteq I_{i,\p_i} \]
    for $1\leq i\leq r$, where $I_{i,\p_i}$ denotes the localization of $I_i$ at $\p_i$.
    Put
    \[ J:=\sum_{i=1}^r \left( \left( I_i+\p_i^{m_i} \right)\prod_{j\neq i } \p_j^{m_j} \right). \]
    Then for each $1\leq k \leq r$ we have
    \[ J_{\p_k} = I_{k,\p_k}. \]
\end{lemma}
\begin{proof}
    Fix an index $k$.
    For $1\leq i\leq r$, by Lemma \ref{lemma:loc_idl_operations}, we have
    \[
    \left( I_i+\p_i^{m_i} \right)_{\p_k} =
    \begin{cases}
        I_{k,\p_k}, & \text{ if }i=k\\
        S_{\p_k}, & \text{ if }i\neq k
    \end{cases}
    \]
    and
    \[
    \left( \prod_{j\neq i } \p_j^{m_j} \right)_{\p_k} =
    \begin{cases}
        S_{\p_k}, & \text{ if }i=k\\
        \p_k^{m_k}, & \text{ if }i\neq k.
    \end{cases}
    \]
    Hence
    \begin{align*}
        J_{\p_k}
        & = \left( \left( I_k+\p_k^{m_k} \right)\prod_{j\neq k } \p_j^{m_j} \right)_{\p_k} +
            \sum_{i\neq k} \left( \left( I_i+\p_i^{m_i} \right)\prod_{j\neq i } \p_j^{m_j} \right)_{\p_k} \\
        & = I_{k,\p_k}S_{\p_k} + S_{\p_k}\p_k^{m_k}\\
        & = I_{k,\p_k} + \p_k^{m_k} \\
        & = I_{k,\p_k},
    \end{align*}
    as required.
\end{proof}

\begin{cor}\label{cor:idlclassestype2}
    Let $S\subseteq T$ be orders such that $\type{T}=2$.
    Let $N$ be the number of primes $\p$ of $T$ such that $\typep{T}=2$.
    Then $\vert\mathcal{W}_T(S)\vert = 2^N$ and $\vert\ICM_T(S)\vert = 2^N\cdot \vert \Pic(T) \vert$.
\end{cor}
\begin{proof}
    The statement follows from Theorem \ref{thm:type2} and Lemma~\ref{lemma:patchlocaltoglobal}.
\end{proof}
\begin{example}\label{ex:appliations}
    Consider the polynomial
    \[ f=x^6 - 6x^5 + 6x^4 + 43x^3 + 96x^2 - 1536x + 4096 \]
    Consider the \'etale $\Q$-algebra $K=\Q[x]/(f)$.
    Denote by $\pi$ the class of the variable $x$ in $K$.
    Put $S=\Z[\pi]$ in $K$.
    There are $3$ non-invertible primes of $S$:
    the unique primes $\p_2,\p_3$ and $\p_5$ above $2,3$ and $5$, respectively.
    The dimension of $\cO_K/\p_i\cO_K$ over $S/\p_i$, for $i=2,3$ and $5$, are $3,3$ and $2$, respectively.
    By Theorem \ref{thm:equiv_conds}, all the overorders of $S$ are either Gorenstein or have type $2$, and there is at least one overorder with type $2$.

    One computes that $S$ has $1728$ overorders, $416$ of which are Gorenstein.
    Using Corollary \ref{cor:idlclassestype2},
    we can compute that there are $3920$ weak equivalences classes, and, after computing all the Picard group, we see that there are $17638816$ isomorphism classes.

    We now exhibit two applications of this computation.
    As described in \cite[Corollary~8.2]{MarsegliaICM}, using a generalization of the Latimer-MacDuffee Theorem \cite{LaClMD33},
    we have a bijection between the conjugacy classes of $\Z$-matrices with characteristic polynomial $f$ and $\ICM(S)$.
    Hence we deduce that there are $17638816$ conjugacy classes of $\Z$-matrices with characteristic polynomial $f$.

    The second application is about abelian varieties over finite fields.
    The polynomial $f$ above defines by Honda-Tate theory an isogeny class of ordinary abelian varieties over the finite field $\mathbb{F}_{16}$.
    See the LMFDB \cite{lmfdb} label \href{https://www.lmfdb.org/Variety/Abelian/Fq/3/16/ag_g_br}{{3.16.ag\_g\_br}}.
    Using \cite[Corollary~4.4]{MarsegliaAbVar}, we see that this isogeny class contains as many isomorphism classes of abelian varieties as the size of $\ICM(\Z[\pi,16/\pi])$, which is a subset of $\ICM(S)$.
    In particular, we check that there are $15112$ such classes.
\end{example}

\section{Comparison with other notions of close-to-being Gorenstein}\
\label{sec:comparison}
In the literature there are several conditions that a ring can satisfy that make it close-to-being Gorenstein.
In this section, when applicable, we will study them in terms of their type.
As usual, we will denote by $Q$ the field of fractions of a Dedekind domain, and by $K$ a non-zero \'etale $Q$-algebra.

\subsection{Nearly Gorenstein}
Let $R$ be a ring and $M$ an $R$-module.
The {\it trace ideal} of $M$ is defined as
\[ \tr M = \sum_{f \in \Hom_R(M,R)} f(M).\]
See for example \cite[Def.~2.1]{Dao21}.
Let $(R, \frm)$ be a Cohen–Macaulay local ring with canonical module $C$.
Then $R$ is called {\it nearly Gorenstein} if $\tr C$ contains $\frm$.
Such a notion was introduced and studied in \cite{Herzogetal19}.

If $S$ is an order in $K$ then by Lemma \ref{lemma:homcolon} we have an identification
\[ \tr S^t = \sum_{f \in (S:S^t)} f\cdot S^t = S^t(S:S^t). \]
Since we are going to work locally, in order to lighten the notation, we set $R=\hat S_\p$ and $\frm=\p \hat S_\p$, where $\p$ is a prime of $S$.
In particular, $R^t$ is the canonical module of $R$ and $\type{R}=\typep{S}$ by Proposition \ref{prop:can_mod_and_type}.
Using this notation, $R$ is nearly Gorenstein if and only if
\[ \frm \subseteq R^t(R:R^t). \]
Observe that the inclusion is strict if and only if $R^t(R:R^t)=R$, that is, $\type{R}=1$, or equivalently $R$ is Gorenstein by Proposition \ref{prop:Gorenstein}.

If $R$ is nearly Gorenstein, but not Gorenstein, then the type of $R$ is
\[ \type{R} = \dim_{R/\frm} \frac{R^t}{(R^t)^2(R:R^t)}. \]
By Lemma \ref{lemma:frac_idls}.\ref{lemma:frac_idls:colontrace}, we get the following equalities of $R/\frm$-vector spaces
\[ \frac{R^t}{(R^t)^2(R:R^t)} = \frac{R^t}{(R:R^t)^t(R:R^t)} = \frac{R^t}{T^t}, \]
where $T$ is the multiplicator ring of $(R:R^t)$.
Similarly, by multiplying $m=R^t(R:R^t)$ by $R^t$, we get that $\frm R^t =T^t$.
Hence $(R:\frm)=T$, which by Lemma \ref{lemma:invmultring} leads to $T=(\frm:\frm)$.
We summarize these observations in the following proposition.
\begin{prop}
    If $R$ is nearly Gorenstein, but not Gorenstein, then
    \[ \type{R} = \dim_{R/\frm} \frac{R^t}{(R^t)^2(R:R^t)}. \]
    Moreover, the multiplicator ring $(\frm:\frm)$ of $\frm$ equals the multiplicator ring of~$(R:R^t)$.
\end{prop}

\subsection{Almost Gorenstein and Weakly almost Gorenstein}
Let $R$ be a local Cohen-Macaulay ring $R$ of Krull dimension $1$ with normalization $\bar{R}$ and conductor $\frf=(R:\bar{R})$.
We say that $R$ is {\it almost Gorenstein} if
\[ \ell_R\left( \frac{\bar R}{R}\right) = \ell_R\left( \frac{R}{\frf}\right) +\type{R} -1, \]
where $\ell_R(-)$ is the length as an $R$-module.
The study of almost Gorenstein rings, where this definition was given, started with \cite{BarucciFroberg97}.
It was extended in \cite{Goteteal13} and \cite{Gotoetal15}.
It is known that almost Gorenstein rings are nearly Gorenstein, see \cite[Prop.~6.1]{Herzogetal19}.

Let $S$ be an order in $K$ and $\p$ a prime of $S$.
We work locally, and again we simplify the notation by
setting $R=\hat S_\p$, $\frm=\p \hat S_\p$, $\bar{R} = \hat \cO_{K,\p}$ and $\frf=(R:\bar R)$.

We can rephrase the definition of the type of $R$ using the length as
\[ \type{R} = \ell_R\left( \frac{R^t}{\frm R^t} \right) = \ell_R\left( \frac{(R:\frm)}{R} \right), \]
where the second equality follow from Lemma \ref{lemma:frac_idls}.\ref{lemma:frac_idls:colontrace} and \ref{lemma:frac_idls:pontryagin}.
Using that the length is an additive function on short exact sequences,
we get the following proposition.
\begin{prop}
    The order $R$ is almost Gorenstein if and only if
    \[ \ell_R\left( \frac{\bar R}{R} \right) = \ell_R\left( \frac{ (R:\frm) }{\frf} \right) -1. \]
\end{prop}

Related to the definition of almost Gorenstein, there is the notion of {\it weakly almost Gorenstein rings } introduced in \cite{Daoetal21}.
In Krull dimension one, the two notions coincide, see \cite[Thm.~1.3]{Daoetal21}.

\bibliographystyle{amsalpha}
\renewcommand{\bibname}{References} 
\providecommand{\bysame}{\leavevmode\hbox to3em{\hrulefill}\thinspace}
\providecommand{\MR}{\relax\ifhmode\unskip\space\fi MR }
\providecommand{\MRhref}[2]{%
  \href{http://www.ams.org/mathscinet-getitem?mr=#1}{#2}
}
\providecommand{\href}[2]{#2}

\end{document}